\documentclass[11pt]{amsart}

\usepackage{txfonts}
\usepackage{amssymb}
\usepackage{graphics}
\usepackage{pst-node}
\usepackage{mathtools}
\usepackage{bbm}
\usepackage{array}

\newtheorem{theorem}{Theorem}[section]
\newtheorem{lemma}[theorem]{Lemma}
\newtheorem*{T1}{Theorem 1}
\newtheorem*{T2}{Theorem 2}

\theoremstyle{definition}

\theoremstyle{remark}

\numberwithin{equation}{section}

\begin{document}

\title{The Expected Perimeter in Eden and Related Growth Processes}

\author{Gabriel Bouch}

\address{Department of Mathematics, Rutgers University, Piscataway, New Jersey 08854}

\begin{abstract}
Following Richardson and using results of Kesten on First-passage percolation, we obtain an upper bound on the expected perimeter in an Eden Growth Process. Using results of the author from a problem in Statistical Mechanics, we show that the average perimeter of the lattice animals resulting from a very natural family of ``growth histories" does not obey a similar bound.
\end{abstract}

\maketitle

\section{Introduction}
An Eden growth process on $\mathbb{Z}^{d}$ is a discrete-time Markov process. The state space at step $n$ is the collection of all connected subsets of $\mathbb{Z}^{d}$ (or \textit{lattice animals}) of size $n+1$ containing the origin. Given a lattice animal $L(n-1)$ containing $n$ lattice points, the possible lattice animals at step $n$ are those which can be realized by adding a lattice point from the perimeter of $L(n-1)$. (Here a lattice point $y \notin L(n-1)$ is on the perimeter of $L(n-1)$ if there exists $x \in L(n-1)$ such that $\{x,y\}$ is a nearest-neighbor pair. We will also call such a nearest-neighbor pair a \textit{perimeter edge}.) In Eden's original formulation (which is the only one we will consider here), the probability of choosing any particular lattice point $y$ on the perimeter is 

\begin{equation}
\frac{\text{number of perimeter edges containing $y$}}{\text{total number of perimeter edges for $L(n-1)$}} \, .
\end{equation}
\\

Computer simulations of a two-dimensional Eden growth process demonstrate that the typical lattice animal containing a large number of lattice points grown by such a method is very nearly a ball \cite{Stauffer}. In addition, if the radius of this ``ball'' is $t$, then simulations also suggest that nearly all of the perimeter sites are contained in a surface layer of thickness $~t^{\frac{1}{3}}$. It is straightforward to turn this estimate of the surface layer thickness into an upper bound on the expected perimeter. Results of Kesten \cite{Kesten} on first-passage percolation and a method of Richardson \cite{Richardson} for associating an Eden growth process with a continuous-time process imply an upper bound on the thickness of the surface layer, and so we are able to demonstrate the following.

\begin{T1}
 The expected perimeter in a $d$-dimensional Eden growth process is bounded above by $Kn^{1-\frac{1}{d(2d+5) +1}}$ for some constant $K$.
\end{T1}

\section{First-passage Percolation and the Eden Growth Process}
We follow \cite{Kesten} for the setup of first-passage percolation, and then mention the key results we will need. Let $(\Omega,\mathcal{F},\mathcal{P})$ be a probability space and $\{ \tau(\{ x,y \}) \, | \, x,y \in \mathbb{Z}^{d}, \text{dist}(x,y)=1 \}$ a collection of independent random variables on $\Omega$ each having the exponential distribution with rate $1$. A \textit{path} $r$ from $x \in \mathbb{Z}^{d}$ to $ y \in \mathbb{Z}^{d}$ is a finite sequence of ordered pairs $(x_{1},x_{2}),(x_{2},x_{3}), \ldots ,(x_{n-1},x_{n})$ such that $x_{1}=x$ and $x_{n}=y$. The passage time of such a path $r$ is

\begin{equation}\label{fppdef1}
 \displaystyle{T(r)(\omega) \coloneqq \sum_{i=1}^{n-1}\tau(\{ x_{i},x_{i+1} \})(\omega)} \, .
\end{equation}
The passage time from $x$ to $y$ is

\begin{equation}\label{fppdef2}
 \displaystyle{T(x,y)(\omega) \coloneqq \inf\{ T(r)(\omega) \, | \, r \text{ is a path from } x \text{ to } y \}} \, ,
\end{equation}
and we define

\begin{align}
 \tilde{B}(t)& = \{ v \, | \, T(0,v) \leq t \}\\
 B(t)& =\left\lbrace  v + x \, | \, v \in \tilde{B}(t), x \in \left[  -\frac{1}{2},\frac{1}{2} \right]^{d} \right\rbrace  \, .
\end{align}

In (Kesten reference) Kesten shows that there exists a compact convex set $B_{0} \subset \mathbb{R}^{d}$ with nonempty interior such that for $t \geq 1$:

\begin{equation}\label{kest1}
 \mathcal{P}\left\lbrace \frac{B(t)}{t} \subset \left( 1 + \frac{x}{\sqrt{t}} \right) B_{0} \right\rbrace \geq 1 - C_{1}t^{2d}e^{-C_{2}x} \hspace{0.1in} \text{if} \hspace{0.1in} x \leq \sqrt{t} \hspace{0.2in} \text{and}
\end{equation}

\begin{multline}\label{kest2}
 \mathcal{P} \left\lbrace \left( 1-C_{3}t^{-\frac{1}{2d+4}}\left( \log t \right)^{\frac{1}{d+2}} \right) B_{0} \subset \frac{B(t)}{t} \right\rbrace \\
  \geq 1 - C_{4}t^{d}\exp \left( -C_{5}t^{\frac{d+1}{2d+4}}\left( \log t \right) ^{\frac{1}{d+2}} \right) \, . 
\end{multline}
Moreover,

\begin{multline}\label{kest3}
 \mathcal{P} \left\lbrace  \left( 1 - 2C_{3}t^{-\frac{1}{2d+4}}\left( \log t \right) ^{\frac{1}{d+2}} \right) B_{0} \subset \frac{B(t)}{t} \right.\\
\left. \subset \left( 1 + C_{6}\frac{\log t}{\sqrt{t}} \right) B_{0} \hspace{0.1in} \text{for all large $t$} \right\rbrace  =1 \, .
\end{multline}
We would like to relate these results on first-passage percolation to the Eden growth process.\\

For almost every $\omega$ it is the case that for every $z \in \mathbb{Z}^{d}$, $\infty > T(0,z)=T(r)$ for some path $r$. (Otherwise, the infimum in the definition of $T(0,z)$ (see \eqref{fppdef2}) is not achieved for some $z$ and, therefore, there are infinitely many paths from $0$ to $z$ with passage time less that $T(0,z)+1$. Hence, $B(T(0,z)+1)$ is not contained in any ball. By \eqref{kest3} this can happen only for $\omega$ in a set of probability $0$. Since $\mathbb{Z}^{d}$ is countable, the result follows.) Also, it is easy to see that for almost every $\omega$, $T(r_{1}) \neq T(r_{2})$ for all pairs of paths such that $r_{1} \neq r_{2}$. Thus, to almost every $\omega$ there exists a unique increasing sequence of times $0 < t_{1} < t_{2} < \ldots$ such that:

\begin{gather}
 \lim_{n \to \infty} t_{n} = \infty \\
 |\tilde{B}(t_{1})| = 2 \\
|\tilde{B}(t_{n}) - \tilde{B}(t_{n-1})|=1 \\
|\tilde{B}(t_{n}) - \lim_{t\uparrow t_{n}}\tilde{B}(t)|=1 \, .
\end{gather}

It is clear that $\tilde{B}(t_{n})$ is a lattice animal $L(n)$ of size $n+1$ and that $t_{n}$ is the smallest $t$ for which $\tilde{B}(t)$ contains $n+1$ points. So, to almost every $\omega$ we can associate a unique sequence of lattice animals $L(0)(\omega),L(1)(\omega),L(2)(\omega),L(3)(\omega), \ldots$ such that:

\begin{enumerate}
 \item $L(0)(\omega)$ is the origin;
 \item $|L(n+1)(\omega) - L(n)(\omega)|=1$;
 \item $L(n+1)(\omega) - L(n)(\omega)$ is a perimeter vertex of $L(n)(\omega)$.
\end{enumerate}
Without loss of generality, we redefine $\Omega$ to be the set containing only those $\omega$ which satisfy the desired properties outlined in the previous two paragraphs.\\

\begin{lemma}
 The sequence of random variables $\{L(n)\}_{n=0}^{\infty}$ is an Eden growth process.
\end{lemma}

\begin{proof}
 We will demonstrate that $L(n+1)$ depends only on $L(n)$ and not on $L(0), \ldots , L(n-1)$ and that $\mathcal{P}(\{ L(n+1)=L(n) \cup \{ v \} \})$ for some $v$ on the perimeter of $L(n)$ is given by 

\begin{equation}\label{probnextv}
\frac{\text{number of perimeter edges containing $v$}}{|p(L(n))|} 
\end{equation}
where $p(L(n))$ is the collection of perimeter edges of $L(n)$.\\

Fix a valid lattice animal evolution $l_{0},l_{1},l_{2}, \ldots , l_{n}$ where $l_{0}$ is the origin. (That is, $l_{j}$ is a lattice animal consisting of $l_{j-1}$ plus a vertex on the perimeter of $l_{j-1}$.) For any lattice animal $L$, we make the following definition.

\begin{equation}
 E^{\text{int}}(L) \coloneqq \{ e=\{ v_{1},v_{2} \} \hspace{0.08in} \text{an edge} \, | \, v_{1},v_{2} \in L \}
\end{equation}
If $|E^{\text{int}}(l_{n})|=N$, then define $X_{i} \coloneqq \tau(e_{i})$, the passage time of the $i^{\text{th}}$ interior edge, $i=1, \ldots ,N$ for some enumeration of the edges in $E^{\text{int}}(l_{n})$. Similarly, if $|p(l_{n})|=M$, then define $Y_{j} \coloneqq \tau(g_{j})$, $j=1, \ldots ,M$ for some enumeration of the edges in $p(l_{n})$. Define $v_{s} \coloneqq l_{s} - l_{s-1}$ for $s=1, \ldots, n$ and $v_{0} \coloneqq 0$ (the origin). Then, the conditions $L(1)(\omega)=l_{1}, \ldots ,L(n)(\omega)=l_{n}$ can be written as conditions on $\textbf{X} \coloneqq (X_{1}, \ldots,X_{N})$ and $\textbf{Y} \coloneqq (Y_{1},\ldots,Y_{M})$.\\

Define

\begin{equation}
 f_{s}(\textbf{X})\coloneqq \min \left\lbrace X_{i_{1}} + \cdots + X_{i_{k}} \, | \, (e_{i_{1}}, \dots ,e_{i_{k}}) \hspace{0.08in} \text{is a path in $l_{s}$ joining $0$ to $v_{s}$} \right\rbrace \, .
\end{equation}
The first set of conditions is

\begin{equation}\label{cond1}
 f_{1}(\textbf{X}) < f_{2}(\textbf{X}) < \ldots < f_{n}(\textbf{X}) \, .
\end{equation}
For each $g_{j} \in p(l_{n})$, we have the additional conditions

\begin{equation}\label{cond2}
 Y_{j}+f_{r_{j}}(\textbf{X}) > f_{n}(\textbf{X}), \text{where } g_{j}=\{v_{r_{j}},w_{j}\}, w_{j} \in p(l_{n}) \, .
\end{equation}

Define $Z_{j} \coloneqq Y_{j}+f_{r_{j}}(\textbf{X})-f_{n}(\textbf{X})$. Let $\tilde{\Omega} \subset \Omega$ be the set where \eqref{cond1} and \eqref{cond2} are satisfied. We want to calculate

\begin{equation*}
 \mathcal{P}(Z_{j} \leq t \, | \, \tilde{\Omega}) \hspace{0.1in} \text{and} \hspace{0.1in} \mathcal{P}(Z_{1} \leq t_{1}, \ldots , Z_{M} \leq t_{M} \, | \, \tilde{\Omega}) \, .
\end{equation*}
By definition

\begin{equation}
 \mathcal{P}(Z_{j} \leq t \, | \, \tilde{\Omega}) = \frac{\mathcal{P}((Z_{j} \leq t) \cap \tilde{\Omega})}{\mathcal{P}(\tilde{\Omega})} = \frac{\mathbb{E}(\chi_{[Z_{j} \leq t] \cap \tilde{\Omega}})}{\mathbb{E}(\chi_{\tilde{\Omega}})} \, .
\end{equation}

Let $A$ be the (Borel) subset of $\mathbb{R}_{+}^{N} \times \mathbb{R}_{+}^{M}$ where the following inequalities are satisfied:

\begin{gather}
 0 < f_{1}(x) < \ldots < f_{n}(x) \\
 0 < y_{j} + f_{r_{j}}(x)-f_{n}(x) \leq t\\
 0 < y_{k} + f_{r_{k}}(x)-f_{n}(x) \hspace{0.1in} \text{for} \hspace{0.08in} k \in \{1, \ldots, M\} - \{j\} \, 
\end{gather}
and let $B=\{x \in \mathbb{R}_{+}^{N} \, | \, 0 < f_{1}(x) < \ldots < f_{n}(x) \}$.
Then,

\begin{align}
 \mathbb{E}(\chi_{[Z_{j} \leq t] \cap \tilde{\Omega}})& =\int_{\mathbb{R}_{+}^{N} \times \mathbb{R}_{+}^{M}} \chi_{A}e^{-x_{1}-\ldots-x_{N}-y_{1}-\ldots-y_{M}} \, dx \, dy\\
& =\int_{\mathbb{R}_{+}^{N}}e^{-x_{1}-\ldots-x_{N}} \left( \int_{\mathbb{R}_{+}^{M}} \chi_{A}(x,y)e^{-y_{1}-\ldots-y_{M}} \, dy \right) \, dx\\
&=\int_{B}e^{-x_{1}-\ldots-x_{N}} \left[  \left( \int_{f_{n}(x)-f_{r_{j}}(x)}^{t+f_{n}(x)-f_{r_{j}}(x)}e^{-y_{j}} \, dy_{j} \right) \left( \prod_{k \neq j} \int_{f_{n}(x)-f_{r_{k}}(x)}^{\infty} e^{-y_{k}} \, dy_{k} \right) \right]  \, dx\\
&=\int_{B}e^{-x_{1}-\ldots-x_{N}} e^{-\sum_{k \neq j}\left[ f_{n}(x) - f_{r_{k}}(x) \right] }\left[ e^{-\left[ f_{n}(x)-f_{r_{j}}(x)\right] } - e^{-\left[ t+f_{n}(x)-f_{r_{j}}(x)\right] }\right] \, dx\\
&=\left( 1-e^{-t} \right) \int_{B}e^{-x_{1}-\ldots-x_{N}} e^{-\sum_{k=1}^{M}\left[ f_{n}(x) - f_{r_{k}}(x) \right] } \, dx
\end{align}

Now let $C$ be the (Borel) subset of $\mathbb{R}_{+}^{N} \times \mathbb{R}_{+}^{M}$ where the following inequalities are satisfied:

\begin{gather}
 0 < f_{1}(x) < \ldots < f_{n}(x) \\
 0 < y_{k} + f_{r_{k}}(x)-f_{n}(x) \hspace{0.1in} \text{for} \hspace{0.08in} k \in \{1, \ldots, M\} \, .
\end{gather}
Then,

\begin{align}
 \mathbb{E}(\chi_{ \tilde{\Omega}})& =\int_{\mathbb{R}_{+}^{N} \times \mathbb{R}_{+}^{M}} \chi_{C}(x,y)e^{-x_{1}-\ldots-x_{N}-y_{1}-\ldots-y_{M}} \, dx \, dy\\
 &=\int_{B} e^{-x_{1}- \ldots - x_{N}} \left( \prod_{k=1}^{M} \int_{f_{n}(x)-f_{r_{k}}(x)}^{\infty} e^{-y_{k}} \, dy_{k} \right) \, dx\\
 &=\int_{B}e^{-x_{1}-\ldots-x_{N}} e^{-\sum_{k=1}^{M}\left[ f_{n}(x) - f_{r_{k}}(x) \right] } \, dx
\end{align}

Thus,

\begin{equation}\label{prob1}
\mathcal{P}(Z_{j} \leq t \, | \, \tilde{\Omega})=\left( 1-e^{-t} \right)  .
\end{equation}

A completely analogous calculation gives

\begin{equation}\label{prob2}
\mathcal{P}(Z_{1} \leq t_{1}, \ldots , Z_{M} \leq t_{M} \, | \, \tilde{\Omega})=\prod_{k=1}^{M} \left( 1-e^{-t_{k}} \right).
\end{equation}
Now note that the random variable $Z_{k}$ is the expected additional waiting time (beyond the time the $(n+1)^{\text{st}}$ site was added to the lattice animal evolution) for a path containing the perimeter edge $g_{k}$ to be traversed. Equations \eqref{prob1} and \eqref{prob2} show that these waiting times are i.i.d. and depend only on $l_{n}$. Thus, if $v$ is a perimeter vertex for $L(n)$, the probability that $L(n+1)=L(n) \cup \{v\}$ is given by \eqref{probnextv}. This proves the lemma.

\end{proof}

\section{An Upper Bound on the Expected Perimeter}
Now that we have this link between first-passage percolation and Eden growth processes, we will follow ideas of Davidson and use Kesten's results to say something about the expected perimeter of the lattice animals in an Eden growth process. The strategy is the following. For large $n$, we will find times $t_{1}$ and $t_{2}$ between which $\tilde{B}(t)$ is overwhelmingly likely to contain $n$ lattice points. We will then show that between these two times, the boundary sites of $\tilde{B}(t)$ are, with probability nearly $1$, contained in a region of very small volume. The desired result will follow.\\

\begin{lemma}
 Let $s_{1}(n)=\inf \left\lbrace s \, | \, sB_{0} \, \, \text{intersects $n$ unit cubes} \right\rbrace $. (Here the unit cubes are centered at points of $\mathbb{Z}^{d}$.) Then there exists a constant $K_{1}$ depending only on $B_{0}$ and the dimension $d$ of the space such that, for large $n$, $\left( s_{1}(n)+K_{1}\right) B_{0}$ contains at least $n$ cells.
\end{lemma}

\begin{proof}
 Let $B_{1}$ be a closed ball centered at the origin and contained in the interior of $B_{0}$. Let $\rho_{1}$ be the radius of $B_{1}$. It is well-known that the map $\Phi : \partial B_{1}\rightarrow \partial B_{0}$ that takes $x \in \partial B_{1}$ to a positive scalar multiple of itself is a bijective Lipschitz continuous map. Let $C$ be the Lipschitz constant. Tessellate $\mathbb{R}^{d}$ with cubes having edges of length $\frac{1}{t}$, where $\frac{1}{t}$ is less than the distance from $B_{1}$ to $\partial B_{0}$, the edges are parallel to the coordinate axes and such that the origin of $\mathbb{R}^{d}$ is the center of some cube. (In other words, just rescale the tiling of $\mathbb{R}^{d}$ referred to in the statement of the lemma by a factor of $\frac{1}{t}$.) We would like to know the smallest $\alpha$ such that $\alpha B_{0}$ contains every cube that intersects $B_{0}$.\\

Let $M_{1}$ be a cube that intersects the boundary of $B_{0}$, but is not contained in $B_{0}$. Suppose $x \in M_{1} \cap B_{0}^{c}$, and let $y \in M_{1} \cap \partial B_{0}$. Then $  \left\lVert  x - y   \right\rVert   < \frac{\sqrt{d}}{t}$ and

\begin{equation}
   \left\lVert  \frac{\rho_{1}}{  \left\lVert  x   \right\rVert  } x - \frac{\rho_{1}}{  \left\lVert  y   \right\rVert  } y   \right\rVert   < \frac{\sqrt{d}}{t} \, .
\end{equation}
Further,

\begin{equation}
 C \frac{\sqrt{d}}{t} \geq  \left\lVert  \Phi\left( \frac{\rho_{1}}{  \left\lVert  x   \right\rVert  } x \right) - \Phi\left( \frac{\rho_{1}}{  \left\lVert  y   \right\rVert  } y \right)   \right\rVert   =   \left\lVert  \Phi\left( \frac{\rho_{1}}{  \left\lVert  x   \right\rVert  } x \right) - y   \right\rVert \, ,
\end{equation}
which implies

\begin{equation}
 \left\lVert x - \Phi\left( \frac{\rho_{1}}{  \left\lVert  x   \right\rVert  } x \right) \right\rVert \leq \frac{\sqrt{d}}{t} + C \frac{\sqrt{d}}{t} = \left( C + 1 \right) \frac{\sqrt{d}}{t} \, .
\end{equation}
Also $\left\lVert \Phi\left( \frac{\rho_{1}}{  \left\lVert  x   \right\rVert  } x \right) \right\rVert > \rho_{1}$. Thus, if $\alpha \Phi\left( \frac{\rho_{1}}{  \left\lVert  x   \right\rVert  } x \right) = x$, then

\begin{align}
 \alpha &= \frac{\left\lVert \Phi\left( \frac{\rho_{1}}{  \left\lVert  x   \right\rVert  } x \right) \right\rVert + \left\lVert x - \Phi\left( \frac{\rho_{1}}{  \left\lVert  x   \right\rVert  } x \right) \right\rVert}{\left\lVert \Phi\left( \frac{\rho_{1}}{  \left\lVert  x   \right\rVert  } x \right) \right\rVert}\\
& \leq 1 + \frac{\left( C + 1 \right) }{\rho_{1}} \frac{\sqrt{d}}{t} \, .
\end{align}

If we now stretch by a factor of $t$, then $t\left( 1 + \frac{\left( C + 1 \right) }{\rho_{1}} \frac{\sqrt{d}}{t} \right) B_{0}$ will contain every unit cube that intersects $tB_{0}$. Letting $K_{1} \coloneqq \frac{(C+1)\sqrt{d}}{\rho_{1}}$, the lemma is proved.
\end{proof}

\begin{T1}
 In an Eden growth process, the expected perimeter of the lattice animal containing $n$ lattice points is bounded above by $Kn^{1-\frac{1}{d(2d+5) +1}}$ for some constant $K$.
\end{T1}

\begin{proof}
 Let $s_{1}$ be the smallest $s$ such that $sB_{0}$ intersects $n$ cells. Then $s_{1}B_{0}$ contains less than $n$ lattice points and, by the previous lemma, $(s_{1}+K_{1})B_{0}$ contains at least $n$ cells, hence, at least $n$ lattice points. So, if $s_{2}$ is such that $\text{vol}(s_{2}B_{0})=n$, then

\begin{equation}\label{s1s2s3}
 s_{1} \leq s_{2} \leq s_{1} + K_{1} \eqqcolon s_{3} \, .
\end{equation}

To simplify the notation we make the following definitions:

\begin{align}
 f(t) &\coloneqq C_{3}t^{-\frac{1}{2d+4}}\left( \log t \right)^{\frac{1}{d+2}}\\
 g(t) &\coloneqq C_{1}t^{2d}e^{-C_{2}t^{\frac{1}{4}}}\\
 h(t) &\coloneqq C_{4}t^{d} \exp \left( -C_{5}t^{\frac{d+1}{2d+4}} \left( \log t \right) ^{\frac{1}{d+2}} \right) \, .
\end{align}
We also set $x$ in \eqref{kest1} equal to $t^{\frac{1}{4}}$. With this inequality in mind, we want to find $t_{1}$ such that

\begin{equation}
 t_{1}\left( 1 + t_{1}^{-\frac{1}{4}} \right) = s_{1} \, .
\end{equation}
It is straightforward to check that, for large enough $s_{1}$,

\begin{equation}\label{t1ineq}
 \displaystyle{s_{1}-s_{1}^{\frac{3}{4}} < t_{1} < s_{1} - s_{1}^{\frac{3}{4}} + \frac{3}{4}s_{1}^{\frac{1}{2}}} \, .
\end{equation}

We would like to adjust \eqref{kest2} slightly to have something that will be easier to calculate with. Choose a constant $\tilde{C}_{3}$ such that

\begin{equation}
 C_{3}t^{-\frac{1}{2d+4}}\left( \log t \right)^{\frac{1}{d+2}} \leq \tilde{C}_{3}t^{-\frac{1}{2d+5}} \hspace{0.08in} \text{for all $t \geq 1$}.
\end{equation}

Then, \eqref{kest2} can be replaced with

\begin{equation}\label{kest2a}
 \mathcal{P} \left\lbrace \left( 1- \tilde{C}_{3}t^{-\frac{1}{2d+5}}\right) B_{0} \subset \frac{B(t)}{t} \right\rbrace \geq 1 - C_{4}t^{d}\exp \left( -C_{5}t^{\frac{d+1}{2d+4}}\left( \log t \right) ^{\frac{1}{d+2}} \right) \, . 
\end{equation}

In light of \eqref{kest2}, we want to find $t_{2}$ such that

\begin{equation}
 t_{2} \left( 1-\tilde{C}_{3}t_{2}^{-\frac{1}{2d+5}} \right) = s_{3} \, .
\end{equation}
Again it is straightforward to check that, for large enough $s_{3}$,

\begin{equation}\label{t2ineq}
 t_{2} < s_{3} + \tilde{C}_{3}s_{3}^{\frac{2d+4}{2d+5}} + 2\left( \tilde{C}_{3} \right) ^{2} \left( \frac{2d+4}{2d+5} \right) s_{3}^{\frac{2d+3}{2d+5}} \, .
\end{equation}

So, with probability at least $ 1-g(t_{1}) -h(t_{2}) $, our continuously evolving lattice animal will have exactly $n$ lattice points at some time between $t_{1}$ and $t_{2}$. Further, 

\begin{equation}\label{annulus1}
 t_{1} \left( 1- \tilde{C}_{3} t_{1}^{-\frac{1}{2d+5}} \right) B_{0} \subset B(t_{1}) \subset t_{1} \left( 1 + t_{1}^{-\frac{1}{4}} \right) B_{0}
\end{equation}
with probability at least $1-g(t_{1})-h(t_{1})$, and

\begin{equation}\label{annulus2}
 t_{2} \left( 1- \tilde{C}_{3} t_{2}^{-\frac{1}{2d+5}} \right) B_{0} \subset B(t_{2}) \subset t_{2} \left( 1 + t_{2}^{-\frac{1}{4}} \right) B_{0}
\end{equation}
with probability at least $1-g(t_{2})-h(t_{2})$. So, both \eqref{annulus1} and \eqref{annulus2} are both satisfied with probability at least $1-g(t_{1})-g(t_{2})-h(t_{1})-h(t_{2})$. For the Eden growth process defined on the same probability space, we can conclude that the lattice animal $L(n-1)$ will contain all the lattice points contained in $t_{1} \left( 1 - \tilde{C}_{3}t_{1}^{-\frac{1}{2d+5}} \right) B_{0}$ and will be contained in $t_{2} \left( 1 + t_{2}^{-\frac{1}{4}} \right) B_{0}$ with probability at least $1-g(t_{1})-g(t_{2})-h(t_{1})-h(t_{2})$.\\

Therefore, for large $n$, with probability at least $1-g(t_{1})-g(t_{2})-h(t_{1})-h(t_{2})$, all of the boundary lattice points of $L(n-1)$ are contained in

\begin{equation}
 \left[ t_{2} \left( 1 + t_{2}^{-\frac{1}{4}} \right)+K_{1}\right] B_{0} \diagdown \left[ t_{1} \left( 1 - \tilde{C}_{3}t_{1}^{-\frac{1}{2d+5}} \right)-K_{1}\right] B_{0} \, .
\end{equation}
So, the number of boundary lattice points in $L(n-1)$ is bounded above by

\begin{equation}
 \text{vol}\left\lbrace \left[ t_{2} \left( 1 + t_{2}^{-\frac{1}{4}} \right)+2K_{1}\right] B_{0} \right\rbrace - \text{vol}\left\lbrace \left[ t_{1} \left( 1 - \tilde{C}_{3}t_{1}^{-\frac{1}{2d+5}} \right)-2K_{1}\right] B_{0} \right\rbrace \, .
\end{equation}

We note that \eqref{s1s2s3} implies that

\begin{equation}
 s_{3} \leq s_{2}+K_{1} = \left[ \frac{n}{\text{vol}(B_{0})} \right] ^{\frac{1}{d}} + K_{1}
\end{equation}
and

\begin{equation}
 s_{1} \geq s_{2} - K_{1} = \left[ \frac{n}{\text{vol}(B_{0})} \right] ^{\frac{1}{d}} - K_{1} \, .
\end{equation}

Using also \eqref{t1ineq} and \eqref{t2ineq} and performing straightforward but somewhat tedious calculations, we find that, for large enough $n$, with probability at least $1-g(t_{1})-g(t_{2})-h(t_{1})-h(t_{2})$, the number of boundary lattice points in $L(n-1)$ is bounded above by

\begin{equation}
 K_{2} \left[ \frac{n}{\text{vol}(B_{0})} \right] ^{1-\frac{1}{d(2d+5)}}
\end{equation}
for some constant $K_{2}$. Therefore, the number of perimeter edges is bounded above by

\begin{equation}
 2d K_{2} \left[ \frac{n}{\text{vol}(B_{0})} \right] ^{1-\frac{1}{d(2d+5)}} \, .
\end{equation}
\\
For any lattice animal in $\mathbb{Z}^{d}$ with $n$ lattice points it is always the case that the number of perimeter edges is bounded above by $2dn$. For large $t$, $g(t) > h(t)$. So, for large $n$,

\begin{equation}
 1-g(t_{1})-g(t_{2})-h(t_{1})-h(t_{2}) > 1 -4g(t_{1}) > 1 - 4e^{-t_{1}^{\frac{1}{5}}} > 1 - 4e^{-\left[ \frac{1}{2} \left( \frac{n}{\text{vol}(B_{0})} \right) ^{\frac{1}{d}} \right] ^{\frac{1}{5}}}\, .
\end{equation}
Thus, for large $n$, the expected perimeter of $L(n-1)$ is bounded above by

\begin{multline}
 1 - 4e^{-\left[ \frac{1}{2} \left( \frac{n}{\text{vol}(B_{0})} \right) ^{\frac{1}{d}} \right] ^{\frac{1}{5}}} \cdot 2dK_{2} \left[ \frac{n}{\text{vol}(B_{0})} \right] ^{1-\frac{1}{d(2d+5)}} + 4e^{-\left[ \frac{1}{2} \left( \frac{n}{\text{vol}(B_{0})} \right) ^{\frac{1}{d}} \right] ^{\frac{1}{5}}} \cdot 2dn\\
< K_{3}n^{1-\frac{1}{d(2d+5) +1}}
\end{multline}
for some constant $K_{3}$. The result follows.
\end{proof}

\section{Average Perimeter Over Lattice Animal Histories}
In the context of a problem in Statistical Mechanics \cite{Bouch}, the author found it natural to consider an average perimeter that is closely related to the expected perimeter in an Eden growth process. Let $e_{1}, e_{2}, \ldots ,e_{n}$, where $e_{j}=\{x_{j}, y_{j} \}$, be a sequence of edges satisfying the following two conditions:

\begin{enumerate}
\item Either $x_{1}$ is the origin or $y_{1}$ is the origin.
\item For $2 \leq j \leq n$, exactly one of $\{x_{j}, y_{j} \}$ is in $\{x_{1}, y_{1}, \ldots , x_{j-1}, y_{j-1} \}$.
\end{enumerate}

We call such a sequence of edges a \textit{lattice animal history of length $n$}. We also define $L(e_{1}, \ldots , e_{n}) \coloneqq \{x_{1}, y_{1}, \ldots , x_{n}, y_{n} \}$, the lattice animal associated with this sequence of edges. What can we say about the average perimeter of $L(e_{1}, \ldots , e_{n})$, where the average perimeter is taken over all lattice animal histories of length $n$? In an Eden growth process, all lattice animal histories of length $n$ are not equally likely. Therefore we cannot expect the average perimeter in an Eden growth process to be the same as the average perimeter over all lattice animal histories. But should we expect this latter average perimeter to obey a bound similar to that found for the average perimeter in an Eden growth process?\\

It seems to the author that there is good reason to think so. First, it is plausible that lattice animals with small perimeters will have many more lattice trees spanning their collection of vertices than those with large perimeters. Consider, for example, the two lattice animals shown below. The first has many spanning trees while the second has only one.\\

\begin{figure}
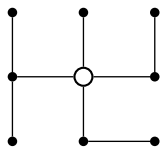

\begin{center}
$
\hspace{0.5cm}
 \begin{psmatrix}[emnode=dot,colsep=.75cm,rowsep=.4cm]
	 & & &\\
	 & [mnode=circle] & &\\
	 & & &\\
\everypsbox{\scriptstyle}
  \ncline[linewidth=.5pt]{1,2}{2,2}
  \ncline[linewidth=.5pt]{1,1}{2,1}
  \ncline[linewidth=.5pt]{1,3}{2,3}
  \ncline[linewidth=.5pt]{2,1}{2,2}
  \ncline[linewidth=.5pt]{2,2}{2,3}
  \ncline[linewidth=.5pt]{2,1}{3,1}
  \ncline[linewidth=.5pt]{2,2}{3,2}
  \ncline[linewidth=.5pt]{3,2}{3,3}
 \end{psmatrix}
$
\caption{A Lattice Animal with ``Small'' Perimeter and Underlying Spanning Tree}
\label{latticeanimal1}
\end{center}
\end{figure}

\begin{figure}
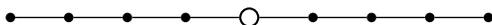

\begin{center}
$
\hspace{0.5cm}
 \begin{psmatrix}[emnode=dot,colsep=.65cm,rowsep=.4cm]
	& & & & [mnode=circle] & & & & &
\everypsbox{\scriptstyle}
  \ncline[linewidth=.5pt]{1,1}{1,2}
  \ncline[linewidth=.5pt]{1,2}{1,3}
  \ncline[linewidth=.5pt]{1,3}{1,4}
  \ncline[linewidth=.5pt]{1,4}{1,5}
  \ncline[linewidth=.5pt]{1,5}{1,6}
  \ncline[linewidth=.5pt]{1,6}{1,7}
  \ncline[linewidth=.5pt]{1,7}{1,8}
  \ncline[linewidth=.5pt]{1,8}{1,9}
 \end{psmatrix}
$
\caption{A Lattice Animal with ``Large'' Perimeter and Underlying Spanning Tree}
\label{latticeanimal2}
\end{center}
\end{figure}

In addition, many of the lattice trees that span lattice animals of small perimeter will ``branch" often. Trees that branch often will have many different orderings of their edges that form valid lattice animal histories. Consider the spanning trees of the lattice animals from the previous figure shown in figure blank. It can be shown (see \cite{Bouch}) that 1680 orderings of the edges in the first tree form lattice tree histories, while 70 orderings of the edges in the second tree form lattice tree histories.\\

With these considerations in mind, it is perhaps surprising to find the following.

\begin{T2}
 No bound of the form $\bar{p}_{n} \leq C_{1}\cdot n^{\alpha}$ with $\alpha < 1$ and $C_{1}$ independent of $n$ exists for the average perimeter taken over all lattice animal histories of length $n-1$.
\end{T2}

The proof of this theorem can be found in the author's companion work \cite{Bouch}.

\bibliographystyle{amsplain}

\end{document}